\documentclass[11pt]{article}
\usepackage[T1]{fontenc}
\usepackage{lmodern}
\usepackage{amsmath,amssymb,amsthm,mathtools}
\usepackage[a4paper,margin=29mm]{geometry}
\usepackage{microtype}
\usepackage{enumitem}
\usepackage{booktabs,array}
\usepackage[colorlinks=true,linkcolor=blue!50!black,citecolor=blue!50!black,urlcolor=blue!50!black]{hyperref}
\usepackage{xcolor}
\hypersetup{pdftitle={Total Rings of Quotients of Prescribed Weak Global Dimension},pdfsubject={A residue-field sequence construction and its homological structure},pdfkeywords={weak global dimension, total quotient ring, non-coherent ring, Tor}}
\usepackage{fancyhdr}
\newtheorem{theorem}{Theorem}[section]
\newtheorem{proposition}[theorem]{Proposition}
\newtheorem{lemma}[theorem]{Lemma}
\newtheorem{corollary}[theorem]{Corollary}
\theoremstyle{definition}
\newtheorem{example}[theorem]{Example}
\theoremstyle{remark}\newtheorem{remark}[theorem]{Remark}
\DeclareMathOperator{\Spec}{Spec}\DeclareMathOperator{\Max}{Max}
\DeclareMathOperator{\Min}{Min}\DeclareMathOperator{\Ann}{Ann}
\DeclareMathOperator{\Tor}{Tor}\DeclareMathOperator{\Ext}{Ext}
\DeclareMathOperator{\fd}{fd}\DeclareMathOperator{\pd}{pd}
\DeclareMathOperator{\wgd}{w.gl.dim}\DeclareMathOperator{\gld}{gl.dim}
\DeclareMathOperator{\Jac}{Jac}\DeclareMathOperator{\Nil}{Nil}
\DeclareMathOperator{\Reg}{Reg}
\newcommand{\N}{\mathbb N}\newcommand{\m}{\mathfrak m}
\newcommand{\p}{\mathfrak p}
\newcommand{\T}{\mathcal T}
\newcommand{\kk}{\kappa}
\numberwithin{equation}{section}
\title{Total Rings of Quotients with Arbitrary Weak Global Dimensions}
\author{Xiaolei Zhang \thanks{School of Mathematics and Statistics, Tianshui Normal University 
		Tianshui 741000, China;	E-mail: zxlrghj@163.com }}
\date{}
\begin{document}
\maketitle
\begin{abstract}
For every nonnegative integer $n$, we construct a commutative total ring of quotients of weak global dimension exactly $n$. For $n\geq1$ the examples are reduced and non-coherent. The construction applies to any nonzero local ring $(A,\m)$: one adjoins countably many residue-field coordinates subject to eventual agreement with the residue of the $A$-coordinate. The resulting ring $\T(A)$ satisfies $Q(\T(A))=\T(A)$ and $\wgd\T(A)=\wgd A$. Its prime spectrum, maximal localizations, nilradical, Jacobson radical, and idempotents are described explicitly. For every pair of modules, positive-degree Tor over $\T(A)$ is naturally identified with Tor over $A$ after quotienting by a projective ideal generated by orthogonal idempotents. We prove that $\T(A)$ is coherent if and only if $A$ is a field. Specializing to regular local polynomial rings yields the prescribed dimensions, with explicit Koszul witnesses for the lower bounds. We also examine the minimal spectrum, the distinction between global and local Pr\"ufer conditions, and a finitely presented cyclic module of flat dimension one and projective dimension two.
\end{abstract}

\section{Introduction}
A commutative ring $R$ is a total ring of quotients when every non-zero-divisor of $R$ is a unit. Equivalently, its canonical embedding into its classical total quotient ring $Q(R)$ is an isomorphism. This condition places a restriction on the regular elements of $R$, but it does not by itself prescribe the homological dimensions of its localizations. The purpose of this paper is to make that distinction explicit and use it to realize arbitrary finite weak global dimensions.

The question we address is the following:
\begin{quote}
For every integer $n\geq0$, does there exist a commutative ring $R$ such that $R=Q(R)$ and $\wgd R=n$?
\end{quote}
We give an affirmative answer by a uniform construction. The examples for positive $n$ are non-coherent. This qualification is essential at dimension zero: a commutative ring of weak global dimension zero is von Neumann regular and therefore coherent.

The construction retains an arbitrary local ring $A$ as one maximal localization and makes every other maximal localization a field. At the same time, infinitely many isolated field coordinates ensure that every nonunit has a nonzero annihilator. The two requirements are compatible because localization can turn a zero divisor into a regular nonunit by annihilating its former annihilators.

\begin{theorem}[Main construction]\label{thm:main}
Let $(A,\m)$ be a nonzero commutative local ring, let $\kk=A/\m$, and write $\bar a$ for the residue of $a\in A$. Define
\[
\T(A)=\left\{(a,(c_i)_{i\in\N})\in A\times\kk^{\N}:
 c_i=\bar a\text{ for all but finitely many }i\right\}.
\]
With coordinatewise operations, $R=\T(A)$ has the following properties.
\begin{enumerate}[label=\textup{(\arabic*)},leftmargin=*]
\item $R=Q(R)$.
\item There is a maximal ideal $M_\infty$ with $R_{M_\infty}\cong A$. All other maximal localizations are isomorphic to $\kk$.
\item $\wgd R=\wgd A$, allowing the value $\infty$, and $\dim R=\dim A$.
\item $R$ is reduced if and only if $A$ is reduced.
\item $R$ is coherent if and only if $A$ is a field.
\end{enumerate}
\end{theorem}

\begin{corollary}[Prescribed finite dimensions]\label{cor:main}
Fix a field $k$. For $n\geq0$, let
\[
 A_n=k[x_1,\ldots,x_n]_{(x_1,\ldots,x_n)},\qquad R_n=\T(A_n),
\]
where $A_0=k$. Then
\[
 Q(R_n)=R_n,\qquad \wgd R_n=\dim R_n=n.
\]
Every $R_n$ is reduced and non-Noetherian. For $n\geq1$, $R_n$ is non-coherent; $R_0$ is coherent and von Neumann regular.
\end{corollary}

The homological statement admits a modulewise refinement. Let $J$ be the ideal of finite-support field coordinates. For an $R$-module $M$, put $M_A=M/JM$. We prove natural identifications
\[
 \Tor_i^R(M,N)\cong\Tor_i^A(M_A,N_A)\quad(i\geq1)
\]
with the right-hand side viewed as an $R$-module through $R\twoheadrightarrow A$. Consequently, $\fd_R M=\fd_A M_A$ when flat dimensions are taken in $\N\cup\{\infty\}$ and the zero module has flat dimension zero. Thus the equality of weak global dimensions is not merely a comparison of upper bounds.

\subsection*{Relation with existing questions}
The distinction between semihereditary, Gaussian, arithmetical, and Pr\"ufer rings is developed in \cite{BG,GS,Glaz}. In the convention used there, a Pr\"ufer ring is a ring whose finitely generated ideals containing a regular element are invertible. Every total ring of quotients satisfies this condition. Glaz and Schwarz ask in \cite[Section 6, Open Question 6]{GS} whether total rings of quotients can have weak global dimensions outside $\{0,1,\infty\}$. The construction proved here gives a negative answer to that formulation, with every finite value occurring. This is a mathematical consequence of the explicit proof; no claim of bibliographic priority is made.

The Gaussian restriction is genuinely stronger. Donadze and Thomas prove the $\{0,1,\infty\}$ restriction for Gaussian rings \cite{DT}. Our examples with $n\geq2$ are non-Gaussian; we exhibit a direct failure of the content formula. This prevents any confusion between the two questions.

The argument does not use an assertion that taking total quotients preserves the weak dimension of a polynomial ring. Instead, it determines every maximal localization of the constructed ring and computes Tor directly. General background on localization, regular local rings, and homological algebra can be found in \cite{AM,CE,Lam,Matsumura,Stacks,Weibel}.

\section{Conventions and homological preliminaries}
All rings are commutative with identity, and all modules are unital. A local ring is assumed nonzero. For a ring $B$, $\Reg(B)$ is the set of elements whose multiplication maps on $B$ are injective, and $Q(B)=\Reg(B)^{-1}B$. We include zero among the zero divisors of a nonzero ring, so that ``every nonunit is a zero divisor'' has its literal usual meaning.

The flat dimension $\fd_B M$ is the least length of a flat resolution of $M$, or $\infty$ if none exists. We use the convention $\fd_B0=0$. The weak global dimension is
\[
 \wgd B=\sup\{\fd_B M:M\text{ is a }B\text{-module}\}.
\]
The projective dimension and global dimension are denoted by $\pd$ and $\gld$. The notations $\Spec B$, $\Max B$, and $\Min B$ designate the prime, maximal, and minimal prime spectra. The latter two carry their subspace Zariski topologies whenever topology is discussed. For $b\in B$, write $D_B(b)=\{\p\in\Spec B:b\notin\p\}$.

\begin{lemma}[Local detection of flat dimension]\label{lem:localfd}
For every $B$-module $M$,
\[
 \fd_B M=\sup_{\p\in\Max B}\fd_{B_\p}M_\p.
\]
In particular,
\[
 \wgd B=\sup_{\p\in\Max B}\wgd B_\p.
\]
\end{lemma}
\begin{proof} It is well-known.
\end{proof}

\begin{lemma}[Flat quotient comparison]\label{lem:basechange}
Suppose $B\to C$ is a flat ring homomorphism. Then
\[
 C\otimes_B\Tor_i^B(M,N)
 \cong\Tor_i^C(C\otimes_BM,C\otimes_BN)
\]
for all $i\geq0$ and all $B$-modules $M,N$.
\end{lemma}
\begin{proof}
Take a free resolution $P_\bullet\to M$. Flatness makes $C\otimes_BP_\bullet$ a free $C$-resolution of $C\otimes_BM$. The complexes obtained by tensoring with $C\otimes_BN$ and by applying $C\otimes_B-$ to $P_\bullet\otimes_BN$ are naturally isomorphic. Flatness also permits tensoring to commute with their homology.
\end{proof}

A ring is \emph{coherent} if every finitely generated ideal is finitely presented. We shall only need the following necessary condition: if $B$ is coherent, then $\Ann_B(b)$ is finitely generated for every $b\in B$. Indeed, the exact sequence
\[
 0\longrightarrow\Ann_B(b)\longrightarrow B\longrightarrow bB\longrightarrow0
\]
has a finitely presented final term, so its kernel is finitely generated.

\begin{lemma}[The zero-dimensional homological case]\label{lem:zero}
A commutative ring of weak global dimension zero is von Neumann regular and hence coherent.
\end{lemma}
\begin{proof} It is well-known.
\end{proof}

\section{The residue-field sequence ring}
Fix a local ring $(A,\m)$ and let $\kk=A/\m$. Put $R=\T(A)$ as in Theorem~\ref{thm:main}, that is, Let $(A,\m)$ be a nonzero commutative local ring, let $\kk=A/\m$, and write $\bar a$ for the residue of $a\in A$. Define
\[
\T(A)=\left\{(a,(c_i)_{i\in\N})\in A\times\kk^{\N}:
c_i=\bar a\text{ for all but finitely many }i\right\}.
\]
With coordinatewise operations.  Eventual equality throughout the paper means equality outside a finite subset of $\N=\{1,2,\ldots\}$. 

\subsection{Basic structure and alternative descriptions}
\begin{proposition}\label{prop:structure}
The set $R$ is a unital subring of $A\times\kk^{\N}$. There is a split exact sequence of rings and additive groups
\[
 0\longrightarrow J\longrightarrow R\xrightarrow{\pi}A\longrightarrow0,
 \qquad J=\bigoplus_{i\in\N}\kk e_i,
\]
where
\[
 \pi(a,(c_i))=a,\qquad
 s(a)=(a,(\bar a,\bar a,\ldots)).
\]
Each $e_i$ is a nonzero idempotent, $e_ie_j=0$ for $i\ne j$, and $e_iR\cong\kk$.
\end{proposition}
\begin{proof}
The union of the exceptional sets for two elements is finite. Outside this union the sum and product have coordinates $\bar a+\bar b=\overline{a+b}$ and $\bar a\bar b=\overline{ab}$. The identity is $(1,(1,1,\ldots))$. The kernel of $\pi$ consists exactly of finite-support sequences, and $s$ is a unital ring section. All assertions about the $e_i$ follow from coordinatewise multiplication; specifically, $(a,(c_j))e_i=c_ie_i$.
\end{proof}

Let $C(\kk)$ denote the ring of eventually constant $\kk$-valued sequences, with $\lambda:C(\kk)\to\kk$ the eventual-value map. The construction is the pullback
\[
 R=A\times_{\kk}C(\kk)
   =\{(a,c):\bar a=\lambda(c)\}.
\]
It also has an additive presentation $A\oplus\kk^{(\N)}$: writing an element as $s(a)+u$, multiplication becomes
\begin{equation}\label{eq:dorroh}
 (a,u)(b,v)=(ab,\bar a v+\bar b u+uv),
\end{equation}
where $uv$ is coordinatewise multiplication of finite-support vectors. The term $uv$ is essential. In particular, $J^2=J$; this ring is not the square-zero idealization of the $A$-module $\kk^{(\N)}$.

\begin{proposition}[Finite stages]\label{prop:stages}
For a finite set $F\subseteq\N$, let
\[
 R_F=\{(a,(c_i))\in R:c_i=\bar a\text{ whenever }i\notin F\}.
\]
Then $R_F\cong A\times\kk^F$ and $R=\bigcup_F R_F$. If $F\subseteq G$, the transition map under these identifications sends
\[
 (a,(c_i)_{i\in F})\longmapsto
 (a,(c_i)_{i\in F},(\bar a)_{i\in G\setminus F}).
\]
\end{proposition}
\begin{proof}
Restriction to the $A$-coordinate and the coordinates indexed by $F$ is an isomorphism. Its inverse fills all remaining positions by $\bar a$. Every element has a finite exceptional set, proving the union assertion and the formula for the transition maps.
\end{proof}

\begin{remark}
The transition maps need not be flat. For example, if $A$ is a local domain with a regular nonunit $a$, the new field coordinate is annihilated by $a$ as an $A$-module and is not flat over $A$. Our proof of the weak-dimension formula does not assume flatness of these transition maps.
\end{remark}

\subsection{Units and zero divisors}
\begin{proposition}\label{prop:units}
For $r=(a,(c_i))\in R$, the following are equivalent:
\begin{enumerate}[label=\textup{(\roman*)}]
\item $r$ is a unit;
\item $c_i\ne0$ for every $i$;
\item $r$ is a non-zero-divisor.
\end{enumerate}
In particular, $R=Q(R)$.
\end{proposition}
\begin{proof}
A unit has a unit image in each field coordinate, so (i) implies (ii). If (ii) holds, eventual agreement gives $\bar a\ne0$. Since $A$ is local, $a$ is a unit. The coordinatewise inverse
\[
 \left(a^{-1},(c_i^{-1})\right)
\]
belongs to $R$, because its coordinates eventually equal $\bar a^{-1}=\overline{a^{-1}}$. Thus (ii) implies (i). Every unit is regular. If (ii) fails, choose $j$ with $c_j=0$; then $re_j=0$ with $e_j\ne0$, so $r$ is not regular. This proves all equivalences.
\end{proof}

\begin{remark}[Why infinitely many coordinates are needed]
If $a\in\m$, then $\bar a=0$. An element of $R$ with first coordinate $a$ therefore has zero field coordinates outside a finite set and has many explicit annihilators. In a finite product $A\times\kk^F$, by contrast, a regular nonunit $a\in A$ together with nonzero entries in all field positions is still a regular nonunit. The finite exceptional-set condition and the infinite index set work together.
\end{remark}

\section{Prime ideals, localizations, and radicals}
For $\p\in\Spec A$, define $P_\p=\pi^{-1}(\p)$. For each $i\in\N$, define
\[
 M_i=\{(a,(c_j))\in R:c_i=0\}.
\]
Also put $M_\infty=P_\m$.

\begin{theorem}[Complete prime spectrum]\label{thm:spec}
The prime ideals of $R$ are precisely
\[
 \{P_\p:\p\in\Spec A\}\ \sqcup\ \{M_i:i\in\N\}.
\]
The first family is the closed subset $V_R(J)$, naturally homeomorphic to $\Spec A$. Each $M_i$ is both minimal and maximal, and is an isolated point of $\Spec R$. Moreover,
\[
 R_{P_\p}\cong A_\p,\qquad R_{M_i}\cong\kk.
\]
The maximal ideals are $M_\infty$ and the $M_i$.
\end{theorem}
\begin{proof}
Primes containing $J$ correspond to primes of $R/J\cong A$, giving the first family and the homeomorphism. Let $P$ be a prime not containing $J$. Some $e_i$ lies outside $P$. Since $e_i(1-e_i)=0$, primality gives $1-e_i\in P$. The map to the $i$th coordinate has kernel $(1-e_i)R$: if $c_i=0$, then $r=(1-e_i)r$. Hence
\[
 R/(1-e_i)R\cong\kk,
\]
and $P=(1-e_i)R=M_i$. The localization at $e_i$ is $e_iR\cong\kk$, so $D_R(e_i)=\{M_i\}$ and $R_{M_i}\cong\kk$. A prime contained in $M_i$ must be $M_i$, either by the classification or by the direct-product decomposition given by $e_i$. Thus $M_i$ is minimal as well as maximal.

Fix $\p\in\Spec A$. If $z\in J$ has finite support $F$, then $e_F=\sum_{i\in F}e_i$ satisfies $(1-e_F)z=0$. Its image in $A$ is $1$, so $1-e_F\notin P_\p$. Consequently $J_{P_\p}=0$, and
\[
 R_{P_\p}\cong (R/J)_{P_\p/J}\cong A_\p.
\]
Since $A$ is local, $P_\m$ is the only maximal ideal in the first family. Notice that $M_i$ is incomparable with every $P_\p$: $e_i\in P_\p\setminus M_i$, while $1-e_i\in M_i\setminus P_\p$.
\end{proof}

\begin{corollary}\label{cor:dimension}
One has $\dim R=\dim A$, including infinite Krull dimension.
\end{corollary}
\begin{proof}
All nontrivial chains of prime ideals lie in $V_R(J)$ and correspond exactly to chains in $\Spec A$. The isolated primes add only chains of length zero. A nonzero ring has nonempty spectrum, so these points do not change the supremum.
\end{proof}

\begin{proposition}\label{prop:radicals}
The nilradical and Jacobson radical are
\[
 \Nil(R)=\{(a,0):a\in\Nil(A)\},\qquad
 \Jac(R)=\{(a,0):a\in\m\}.
\]
In particular, $R$ is reduced if and only if $A$ is reduced. Every idempotent of $R$ is either $e_F$ or $1-e_F$ for a finite subset $F\subseteq\N$.
\end{proposition}
\begin{proof}
A nilpotent has nilpotent first coordinate and zero image in every field coordinate. Conversely a nilpotent $a\in A$ lies in $\m$, so $(a,0)\in R$ and is nilpotent. For the Jacobson radical, intersect the maximal ideals listed in Theorem~\ref{thm:spec}: membership in every $M_i$ forces all field coordinates to vanish, and the eventual-value condition then forces $a\in\m$.

If $(a,(c_i))$ is idempotent, the local ring $A$ has $a=0$ or $a=1$. Each $c_i$ is $0$ or $1$. If $a=0$, there are only finitely many coordinates equal to $1$, giving $e_F$. If $a=1$, there are only finitely many coordinates equal to $0$, giving $1-e_F$.
\end{proof}

\begin{proposition}[Two contrasting spectral subspaces]\label{prop:topology}
The maximal spectrum of $R$ is homeomorphic to the one-point compactification of the discrete space $\N$, with $M_\infty$ corresponding to the added point. If $A$ is a local domain which is not a field, then $\Min R$ is an infinite discrete, non-quasi-compact space.
\end{proposition}
\begin{proof}
Each $M_i$ is isolated. A basic open neighborhood $D_R(r)$ of $M_\infty$ has first coordinate outside $\m$. Therefore its field coordinates are eventually nonzero, and the neighborhood contains all but finitely many $M_i$. Conversely, for finite $F$,
\[
 D_R(1-e_F)\cap\Max R
 =\{M_\infty\}\cup\{M_i:i\notin F\}.
\]
These are precisely the neighborhoods in the stated compactification.

If $A$ is a domain, the minimal primes are $J=P_{(0)}$ and all $M_i$. Choose $0\ne a\in\m$ and put $t=(a,0)$. Then
\[
 D_R(t)\cap\Min R=\{J\},\qquad D_R(e_i)\cap\Min R=\{M_i\}.
\]
Thus all points of $\Min R$ are open. The cover by these singletons has no finite subcover.
\end{proof}

\section{The projective support ideal and the Tor comparison}
\subsection{A flat quotient which is a localization}
\begin{proposition}\label{prop:flatquotient}
The ideal $J$ is projective and pure in $R$. The quotient $A=R/J$ is a flat $R$-module. More explicitly, if
\[
 S=\{1-e_F:F\subseteq\N\text{ finite}\},
\]
then $S$ is multiplicatively closed and $S^{-1}R\cong A$.
\end{proposition}
\begin{proof}
Each $e_iR$ is a direct summand of $R$, and
\[
 J=\bigoplus_{i\in\N}e_iR.
\]
It is therefore projective. The identity
\[
 (1-e_F)(1-e_G)=1-e_{F\cup G}
\]
shows that $S$ is multiplicatively closed. Projection to $A$ sends every member of $S$ to $1$, and hence induces $S^{-1}R\to A$. It is surjective. If a fraction maps to zero, its numerator belongs to $J$ and is killed by $1-e_F$ for a suitable finite $F$. Thus the map is injective. Localization is flat, proving flatness of $A$ and purity of the kernel in $0\to J\to R\to A\to0$.
\end{proof}

\begin{remark}
The section $s:A\to R$ is a ring section. It does not say that $R\to A$ splits as an $R$-module map. In fact, we shall prove that the cyclic flat $R$-module $A$ is not projective. This distinction is important in every homological use of the construction.
\end{remark}

For an $R$-module $M$, define $M_A=A\otimes_RM=M/JM$ and $M_i=e_iM$. The latter is a vector space over $e_iR\cong\kk$. The identifications
\[
 M_{M_\infty}\cong M_A,\qquad M_{M_i}\cong e_iM
\]
follow from Theorem~\ref{thm:spec}. The first also follows from the localization description of $A$ and the fact that $A$ is local.

\begin{theorem}[Positive-degree Tor comparison]\label{thm:tor}
For all $R$-modules $M,N$ and all $i\geq1$, there is a natural isomorphism
\begin{equation}\label{eq:tor}
 \Tor_i^R(M,N)\cong\Tor_i^A(M_A,N_A),
\end{equation}
where the right-hand side is viewed as an $R$-module through $\pi$. In particular,
\begin{equation}\label{eq:fd}
 \fd_RM=\fd_A M_A
\end{equation}
for every $M$, with our convention for the zero module, and
\begin{equation}\label{eq:wd}
 \wgd R=\wgd A.
\end{equation}
\end{theorem}
\begin{proof}
Let $T_i=\Tor_i^R(M,N)$ for $i\geq1$. Localizing at $e_j$ gives
\[
 e_jT_i\cong\Tor_i^{e_jR}(e_jM,e_jN)
          \cong\Tor_i^{\kk}(e_jM,e_jN)=0.
\]
Hence $JT_i=0$, so the natural quotient map $T_i\to A\otimes_RT_i$ is an isomorphism. Applying Lemma~\ref{lem:basechange} to the flat map $R\to A$ now gives \eqref{eq:tor}.

For any fixed $d\geq0$, if $\fd_A M_A\leq d$, the right-hand side of \eqref{eq:tor} vanishes for $i=d+1$ and all $N$, giving $\fd_RM\leq d$. Conversely, if $\fd_RM\leq d$, let $L$ be any $A$-module, viewed as an $R$-module through $\pi$. Then $L_A=L$, and \eqref{eq:tor} implies $\Tor_{d+1}^A(M_A,L)=0$. This proves \eqref{eq:fd}, including infinite dimensions. Finally, every $A$-module arises as $M_A$ by inflation, so taking suprema proves \eqref{eq:wd}.
\end{proof}

\begin{corollary}\label{cor:flatcriterion}
An $R$-module $M$ is flat if and only if $M/JM$ is flat over $A$. For every $M$, the submodule $JM$ is projective and
\[
 JM=\bigoplus_{i\in\N}e_iM.
\]
\end{corollary}
\begin{proof}
The first statement is \eqref{eq:fd}. Every element of $JM$ is a finite sum of elements from the $e_iM$. The sum is direct, as multiplication by $e_j$ extracts its $j$th summand. Each $e_iM$ is a $\kk$-vector space, hence is a direct sum of copies of the projective module $e_iR$. Arbitrary direct sums of projectives are projective.
\end{proof}

\begin{proposition}[Ext for inflated modules]\label{prop:ext}
For $A$-modules $L,N$ viewed as $R$-modules through $\pi$, there are natural isomorphisms
\[
 \Ext_R^i(L,N)\cong\Ext_A^i(L,N)\qquad(i\geq0).
\]
\end{proposition}
\begin{proof}
Let $P_\bullet\to L$ be a projective $R$-resolution. Since $A$ is flat over $R$, $A\otimes_RP_\bullet$ is a projective $A$-resolution of $A\otimes_RL=L$. The tensor--Hom adjunction identifies the cochain complexes
\[
 \operatorname{Hom}_R(P_\bullet,N)
 \cong\operatorname{Hom}_A(A\otimes_RP_\bullet,N).
\]
Their cohomology gives the result. This formula concerns targets annihilated by $J$; it does not imply equality of projective dimensions over $R$ and $A$.
\end{proof}

\section{Coherence and projective dimensions}
\begin{theorem}[Exact coherence criterion]\label{thm:coherence}
The ring $\T(A)$ is coherent if and only if $A$ is a field. In all cases $\T(A)$ is non-Noetherian.
\end{theorem}
\begin{proof}
The ideal $J$ is never finitely generated. A finite set of its elements is supported on a common finite set $F$, and every linear combination over $R$ remains supported on $F$. Such elements cannot generate $e_j$ for $j\notin F$. Thus $R$ is non-Noetherian.

Suppose $A$ is not a field. Choose $0\ne a\in\m$, and put $t=(a,0)\in R$. An element $(b,(d_i))$ annihilates $t$ precisely when $ab=0$. Since $a\ne0$, every such $b$ lies in $\m$: a unit cannot annihilate a nonzero element. Consequently every annihilator has a finite-support field sequence. The ideal $\Ann_R(t)$ contains all $e_i$. If it had finitely many generators, their field supports would have a finite union $F$, and every element of the generated ideal would have field support contained in $F$. This contradicts $e_j\in\Ann_R(t)$ for $j\notin F$. Therefore $R$ is not coherent. This argument does not require $a$ to be regular in $A$.

If $A$ is a field, take $r=(a,(c_i))$ and define $b=0$ when $a=0$, and $b=a^{-1}$ otherwise. Define $d_i=0$ when $c_i=0$, and $d_i=c_i^{-1}$ otherwise. The pair $u=(b,(d_i))$ belongs to $R$ and satisfies $r=r^2u$. Thus $R$ is von Neumann regular. Its finitely generated ideals are idempotent-generated, hence finitely presented, as in Lemma~\ref{lem:zero}.
\end{proof}

\begin{proposition}\label{prop:pdA}
The $R$-module $A=R/J$ is flat, cyclic, not finitely presented, and has projective dimension exactly one.
\end{proposition}
\begin{proof}
Flatness was proved in Proposition~\ref{prop:flatquotient}. Since $J$ is not finitely generated, the cyclic quotient $R/J$ is not finitely presented. If it were projective, $0\to J\to R\to A\to0$ would split, making $J$ a direct summand ideal of $R$. Such an ideal is generated by one idempotent, contradicting the failure of finite generation. Finally, the same exact sequence has projective first two terms, so $\pd_RA\leq1$. Nonprojectivity gives equality.
\end{proof}

\begin{proposition}[An explicit cyclic module]\label{prop:cyclic}
Suppose $A$ contains a regular nonunit $a$. Let $t=(a,0)$ and $C=R/tR$. Then
\[
 \fd_RC=1,\qquad \pd_RC=2.
\]
The module $C$ is finitely presented, but its first syzygy $tR$ is not finitely presented.
\end{proposition}
\begin{proof}
Since $a$ is regular in $A$, $\Ann_R(t)=J$. Therefore $tR\cong A$ as $R$-modules, with the action on $A$ through $\pi$. Proposition~\ref{prop:pdA} shows that $tR$ is flat, nonprojective, and of projective dimension one. In particular,
\[
 0\longrightarrow J\longrightarrow R\xrightarrow{\cdot t}R
   \longrightarrow C\longrightarrow0
\]
is a projective resolution of length two. If $\pd_RC\leq1$, then the kernel $tR$ of the displayed surjection $R\to C$ would be projective, by dimension shifting (or Schanuel's lemma). Hence $\pd_RC=2$.

The module $C$ is finitely presented because it has one generator and the single relation $t$. Its localization at $M_\infty$ is $A/aA$. The sequence $0\to A\xrightarrow{\cdot a}A\to A/aA\to0$ and regularity of $a$ give
\[
 \Tor_1^A(A/aA,A/aA)\cong A/aA\ne0.
\]
Thus $\fd_A(A/aA)=1$, and Theorem~\ref{thm:tor} gives $\fd_RC=1$. Finally $tR\cong A$ is not finitely presented by Proposition~\ref{prop:pdA}.
\end{proof}

\begin{proposition}[Global-dimension bounds]\label{prop:global}
If $\gld A=d<\infty$, then
\[
 \max\{d,1\}\leq\gld R\leq d+1.
\]
In particular, $\gld\T(k)=1$ for a field $k$. If $A$ is a discrete valuation domain which is not a field, then $\gld\T(A)=2$.
\end{proposition}
\begin{proof}
The lower bound $d$ follows by localizing projective resolutions at $M_\infty$; every $A$-module can be regarded as an $R$-module. The lower bound $1$ follows from $\pd_RA=1$.

For the upper bound, an $A$-projective module is a direct summand of a direct sum of copies of $A$, and hence has $R$-projective dimension at most one. An $A$-projective resolution of an $A$-module $L$ of length at most $d$, together with successive dimension shifting, gives $\pd_RL\leq d+1$. For any $R$-module $M$, the exact sequence
\[
 0\longrightarrow JM\longrightarrow M\longrightarrow M_A\longrightarrow0
\]
has projective first term by Corollary~\ref{cor:flatcriterion}. Therefore $\pd_RM\leq\max\{0,\pd_RM_A\}\leq d+1$. The field case follows immediately. A discrete valuation domain has global dimension one, so the upper bound is two; Proposition~\ref{prop:cyclic}, applied to a uniformizer, attains it.
\end{proof}

\section{Realization of all nonnegative integers}
Fix a field $k$ and let
\[
 B_n=k[x_1,\ldots,x_n],\qquad
 A_n=(B_n)_{(x_1,\ldots,x_n)}.
\]
For $n=0$, these both mean $k$. Set $R_n=\T(A_n)$. We now give an explicit proof of the numerical assertion rather than only invoking the general construction.

\begin{lemma}\label{lem:An}
For every $n\geq0$, $\wgd A_n=n$. For $n\geq1$, the residue field $k$ satisfies
\[
 \Tor_i^{A_n}(k,k)\cong\bigwedge\nolimits_k^i k^n
 \quad(0\leq i\leq n),
\]
and the higher Tor groups vanish.
\end{lemma}
\begin{proof}
Hilbert's syzygy theorem gives $\gld B_n=n$; equivalently, every $B_n$-module has a projective resolution of length at most $n$. Localization gives $\gld A_n\leq n$, and thus $\wgd A_n\leq n$.

For $n\geq1$, the sequence $x_1,\ldots,x_n$ is regular in $A_n$. Indeed, modulo its first $j$ members, the quotient is the corresponding localized polynomial domain in the remaining variables, and the next variable is nonzero. Its Koszul complex $K_\bullet(x_1,\ldots,x_n;A_n)$ is a finite free resolution of $k$. After tensoring with $k$, all differentials become zero because their entries are among the $x_i$. Its $i$th term becomes $\bigwedge_k^i k^n$, proving the formula. In particular $\Tor_n^{A_n}(k,k)\cong k\ne0$. This supplies the opposite inequality. The case $n=0$ is the field case.
\end{proof}

\begin{theorem}[Realization theorem]\label{thm:realization}
For every integer $n\geq0$,
\[
 R_n=\left\{(a,(c_i))\in A_n\times k^{\N}:
            c_i=a\bmod(x_1,\ldots,x_n)\text{ eventually}\right\}
\]
is a reduced total ring of quotients with
\[
 \wgd R_n=\dim R_n=n.
\]
For $n\geq1$ it is non-coherent. For $n=0$ it is coherent and von Neumann regular. Every $R_n$ is non-Noetherian.
\end{theorem}
\begin{proof}
The total quotient property is Proposition~\ref{prop:units}. Reducedness follows from the fact that $A_n$ is a domain and Proposition~\ref{prop:radicals}. The weak-dimension formula follows from Theorem~\ref{thm:tor} and Lemma~\ref{lem:An}. The Krull-dimension formula follows from Corollary~\ref{cor:dimension} and $\dim A_n=n$. The coherence and Noetherian assertions follow from Theorem~\ref{thm:coherence}.
\end{proof}

\begin{corollary}[Explicit witnesses in every positive degree]\label{cor:koszulR}
Regard $k$ as an $R_n$-module through $R_n\to A_n\to k$. Then, for $1\leq i\leq n$,
\[
 \Tor_i^{R_n}(k,k)\cong k^{\binom ni}.
\]
In particular $\fd_{R_n}k=n$ for $n\geq1$.
\end{corollary}
\begin{proof}
Apply the positive-degree Tor comparison and the Koszul calculation. The upper bound on flat dimension is Theorem~\ref{thm:realization}; the nonzero top Tor group proves equality.
\end{proof}

\subsection{The cases zero, one, and two}
For $n=0$, $R_0$ is isomorphic to the ring of eventually constant $k$-valued sequences: the first coordinate is recovered from the eventual value. It is a non-Noetherian von Neumann regular ring. Taking $R=k$ would also answer the realization question at zero, but $R_0$ keeps the construction uniform. No non-coherent example at weak dimension zero can exist by Lemma~\ref{lem:zero}.

For $n=1$, the ring $A_1=k[x]_{(x)}$ is a discrete valuation domain. Hence $R_1$ is a reduced, non-coherent total ring of quotients of weak global dimension one and global dimension two. Its principal ideal generated by $(x,0)$ is flat but not projective.

For $n=2$, one obtains
\[
 R_2=\left\{(a,(c_i))\in k[x,y]_{(x,y)}\times k^{\N}:
                  c_i=a\bmod(x,y)\text{ eventually}\right\}.
\]
The complete list of maximal localizations is
\[
 (R_2)_{M_\infty}\cong k[x,y]_{(x,y)},\qquad
 (R_2)_{M_i}\cong k.
\]
The dimension-two lower bound is the concrete calculation
\[
 \Tor_2^{R_2}(k,k)\cong\Tor_2^{k[x,y]_{(x,y)}}(k,k)\cong k.
\]
The non-coherence witness is equally concrete:
\[
 \Ann_{R_2}((x,0))=J=\bigoplus_{i\geq1}ke_i,
\]
which is not finitely generated. These two calculations are independent: non-coherence does not by itself establish the weak dimension, and the Tor calculation does not replace the annihilator argument.

\begin{remark}[Countable examples]
If $k$ is countable, then each $A_n$ is countable. A sequence occurring in $R_n$ is specified by its first coordinate, a finite exceptional set, and finitely many field values. Hence every $R_n$ is countable. The construction therefore realizes every finite weak dimension even among countable total rings of quotients.
\end{remark}

\subsection{An infinite-dimensional extension}
\begin{example}\label{ex:infinite}
Let
\[
 A_\infty=k[x_1,x_2,\ldots]_{(x_1,x_2,\ldots)}.
\]
For every $r\geq1$, the first $r$ variables form a regular sequence. If $L_r=A_\infty/(x_1,\ldots,x_r)$, the Koszul resolution gives
\[
 \Tor_r^{A_\infty}(L_r,L_r)\cong L_r\ne0.
\]
Thus $\wgd A_\infty=\infty$. It follows that $\T(A_\infty)$ is a reduced non-coherent total ring of quotients of infinite weak global dimension. This extends the realization range to $\N\cup\{\infty\}$.
\end{example}

\section{Pr\"ufer conditions and localization}
\subsection{Semi-regular ideals and finite annihilators}
We use the convention that a \emph{regular ideal} if it contains a non-zero-divisor; a   \emph{semi-regular ideal} if it contains a finitely generated sub-ideal whose annihilator is zero. 
A ring is Pr\"ufer (resp., strongly Pr\"ufer) if each finitely generated regular (resp., semi-regular) ideal is projective. This is the convention in \cite{BG,GS}.

\begin{theorem}\label{thm:strong-prufer}
Every proper finitely
	generated ideal of $\T(A)$  has a nonzero annihilator. Consequently, the only
	semi-regular ideal of $\T(A)$  is $\T(A)$  itself, and $\T(A)$  is a strong
	Pr\"ufer ring, and thus a Pr\"ufer ring.
\end{theorem}

\begin{proof}
	Put $R=T(A)$. For each $i\geq1$, let $e_i\in R$ be the element whose first
	coordinate is zero and whose sequence has value $1$ in position $i$ and $0$
	in every other position. Then $e_i\ne0$, $e_i^2=e_i$, and
	$e_ie_j=0$ when $i\ne j$. Let
	\[
	J=\bigoplus_{i\geq1}\kappa e_i
	\]
	be the ideal of finite-support sequences with first coordinate zero. The
	projection $\pi:R\to A$ onto the first coordinate is a surjective ring
	homomorphism with kernel $J$.
	
	We first identify the maximal ideals needed in the argument. For
	$i\geq1$, set
	\[
	\mathfrak M_i=\ker\bigl(R\longrightarrow\kappa,\ (a,(c_j))\longmapsto c_i\bigr),
	\qquad
	\mathfrak M_\infty=\pi^{-1}(\mathfrak m).
	\]
	Each is maximal, since the corresponding quotient is $\kappa$. Conversely,
	let $\mathfrak M$ be any maximal ideal of $R$. If $J\subseteq\mathfrak M$,
	then $\mathfrak M/J$ is a maximal ideal of
	$R/J\cong A$. As $A$ is local, $\mathfrak M/J=\mathfrak m/J$, and hence
	$\mathfrak M=\mathfrak M_\infty$. If $J\nsubseteq\mathfrak M$, there is an
	$i$ with $e_i\notin\mathfrak M$. Since $e_i(1-e_i)=0$ and $\mathfrak M$ is
	prime, $1-e_i\in\mathfrak M$. The natural map
	$R\to e_iR\cong\kappa$ has kernel $(1-e_i)R=\mathfrak M_i$; maximality
	therefore gives $\mathfrak M=\mathfrak M_i$. Thus
	\[
	\Max(R)=\{\mathfrak M_\infty\}\cup\{\mathfrak M_i:i\geq1\}.
	\]
	
	Now let $I=(r_1,\ldots,r_s)$ be a proper finitely generated ideal of $R$.
	Choose a maximal ideal containing $I$. If $I\subseteq\mathfrak M_i$ for
	some $i$, then each generator has zero $i$th sequence coordinate, so
	$e_ir_\ell=0$ for every $\ell$. It follows that
	$0\ne e_i\in\Ann_R(I)$.
	
	It remains to consider the case $I\subseteq\mathfrak M_\infty$. Write
	$r_\ell=(a_\ell,(c_{\ell i})_{i\geq1})$. Membership in
	$\mathfrak M_\infty$ gives $a_\ell\in\mathfrak m$. By the definition of
	$R$, the sequence $(c_{\ell i})_i$ is eventually equal to
	$\overline{a_\ell}=0$. Thus each such sequence has finite support. Since
	there are only finitely many generators, the union
	\[
	F=\bigcup_{\ell=1}^s\{i:c_{\ell i}\ne0\}
	\]
	is finite. Choose $j\notin F$. Coordinatewise multiplication gives
	$e_jr_\ell=0$ for every $\ell$, and hence $e_jI=0$. Since $e_j\ne0$,
	$\Ann_R(I)\ne0$. This proves the first assertion.
	
	For completeness, call an ideal $L$ dense if $\Ann_R(L)=0$, and
	semi-regular if it contains a finitely generated dense ideal. If $L$ is
	semi-regular, it contains a finitely generated ideal $I$ with
	$\Ann_R(I)=0$. The first assertion forces $I$ not to be proper, so $I=R$;
	therefore $L=R$. Conversely, $R$ is dense because $\Ann_R(R)=0$. Thus $R$
	is the only semi-regular ideal. In particular, every finitely generated
	semi-regular ideal is $R$, which is principal after localization at every
	maximal ideal. By the standard definition, $R$ is a strong Pr\"ufer ring, and thus a  Pr\"ufer ring..
\end{proof}

\subsection{Failure of Gaussianity in dimension at least two}
For a polynomial $f\in B[T]$, let $c_B(f)$ be the ideal generated by its coefficients. A ring is Gaussian if $c_B(fg)=c_B(f)c_B(g)$ for all polynomials $f,g$.

\begin{proposition}\label{prop:nongaussian}
For $n\geq2$, the ring $R_n$ is not Gaussian. It is also not locally Pr\"ufer and not arithmetical.
\end{proposition}
\begin{proof}
Set $u=(x_1,0)$ and $v=(x_2,0)$ in $R_n$, and consider
\[
 f=u+vT,\qquad g=u-vT.
\]
Then
\[
 c_{R_n}(f)c_{R_n}(g)=(u^2,uv,v^2),\qquad
 c_{R_n}(fg)=(u^2,v^2).
\]
These identities hold in every characteristic, including characteristic two. If the two ideals were equal, localization at $M_\infty$ would imply
\[
 x_1x_2\in(x_1^2,x_2^2)A_n.
\]
This is false. Indeed, send all other variables to zero and reduce modulo $(x_1^2,x_2^2)$. This defines a map from $A_n$ to the local Artinian ring $k[x_1,x_2]/(x_1^2,x_2^2)$, since every denominator with nonzero constant term maps to a unit. The image of $x_1x_2$ is a nonzero basis monomial, whereas the images of $x_1^2$ and $x_2^2$ vanish.

The localization at $M_\infty$ is the domain $A_n$. Its ideal $(x_1,x_2)$ is not principal: its quotient by $\m_n(x_1,x_2)$ has $k$-dimension two, as can be checked from the distinct linear monomials. A finitely generated invertible ideal over a local ring is principal, so $A_n$ is not Pr\"ufer. Also $(x_1)$ and $(x_2)$ are incomparable, so $A_n$ is not a valuation domain and its ideals are not linearly ordered. Thus $R_n$ is not locally Pr\"ufer or arithmetical.
\end{proof}

\begin{remark}[The Gaussian theorem remains unaffected]
The theorem of Donadze--Thomas \cite{DT} restricts the weak global dimensions of Gaussian rings. Proposition~\ref{prop:nongaussian} shows directly why it does not apply to $R_n$ for $n\geq2$. The present construction concerns the larger class of total rings of quotients and hence of Pr\"ufer rings under the regular-ideal definition.
\end{remark}

\begin{proposition}[How regularity changes under localization]\label{prop:regularity}
For $n\geq1$, let $u=(x_1,0)\in R_n$. Then $u$ is a zero divisor of $R_n$, but its image in $(R_n)_{M_\infty}\cong A_n$ is a regular nonunit. In particular, a localization of a total ring of quotients need not be a total ring of quotients.
\end{proposition}
\begin{proof}
The nonzero idempotents $e_i$ annihilate $u$. They all vanish after localization at $M_\infty$, because $(1-e_i)e_i=0$ and $1-e_i\notin M_\infty$. The image of $u$ is $x_1$, which is a regular nonunit of the domain $A_n$. Thus $A_n\ne Q(A_n)$.
\end{proof}

\begin{remark}
The statement in \cite[Section 6, Open Question 6]{GS} asks whether $R=Q(R)$ forces $\wgd R\in\{0,1,\infty\}$. The ring $R_2$ has $R_2=Q(R_2)$ and $\wgd R_2=2$, so it disproves that statement as formulated. The realization theorem supplies every larger finite value as well. This conclusion uses only the elementary construction and the Tor and localization calculations established above. The historical reference identifies the formulation being addressed; it is not a hypothesis in the proof, and it is not a claim that no subsequent literature has discussed the same construction.
\end{remark}


\begin{thebibliography}{99}
\bibitem{AM} M. F. Atiyah and I. G. Macdonald, \emph{Introduction to Commutative Algebra}, Addison--Wesley, Reading, MA, 1969.
\bibitem{BG} S. Bazzoni and S. Glaz, Gaussian properties of total rings of quotients, \emph{Journal of Algebra} \textbf{310} (2007), no.~1, 180--193.

\bibitem{CE} H. Cartan and S. Eilenberg, \emph{Homological Algebra}, Princeton University Press, Princeton, NJ, 1956.
\bibitem{DT} G. Donadze and V. Z. Thomas, Bazzoni--Glaz conjecture, \emph{Journal of Algebra} \textbf{420} (2014), 141--160.

\bibitem{Glaz} S. Glaz, \emph{Commutative Coherent Rings}, Lecture Notes in Mathematics, vol.~1371, Springer-Verlag, Berlin, 1989.
\bibitem{GS} S. Glaz and R. Schwarz, Pr\"ufer conditions in commutative rings, \emph{Arabian Journal for Science and Engineering} \textbf{36} (2011), 967--983.

\bibitem{Lam} T. Y. Lam, \emph{Lectures on Modules and Rings}, Graduate Texts in Mathematics, vol.~189, Springer-Verlag, New York, 1999.
\bibitem{Matsumura} H. Matsumura, \emph{Commutative Ring Theory}, translated by M. Reid, Cambridge Studies in Advanced Mathematics, vol.~8, Cambridge University Press, Cambridge, 1986.
\bibitem{Stacks} The Stacks Project Authors, \emph{The Stacks Project}, \url{https://stacks.math.columbia.edu}.
\bibitem{Weibel} C. A. Weibel, \emph{An Introduction to Homological Algebra}, Cambridge Studies in Advanced Mathematics, vol.~38, Cambridge University Press, Cambridge, 1994.
\end{thebibliography}
\end{document}